\documentclass[a4paper,10pt]{article}

\usepackage[english]{babel}
\usepackage{amsmath,amsfonts,amssymb,amsthm}
\usepackage{mathrsfs}
\usepackage{bbm}
\usepackage{indentfirst}

\usepackage{mathtools}
\mathtoolsset{showonlyrefs}

\usepackage{graphicx}
\usepackage{pgf,tikz}
\usepackage[font=small]{caption}

\newtheorem{theorem}{Theorem}[section]

\newtheorem{lemma}{Lemma}[section]
\newtheorem{proposition}{Proposition}[section]

\theoremstyle{definition}
\newtheorem{remark}{Remark}[section]

\numberwithin{equation}{section}

\newcommand\blfootnote[1]{\begingroup\renewcommand\thefootnote{}\footnote{#1}\addtocounter{footnote}{-1}\endgroup}

\begin{document}

\title{
{\bf\Large Positive subharmonic solutions \\ to superlinear ODEs with indefinite weight}
\footnote{Work partially supported by the Grup\-po Na\-zio\-na\-le per l'Anali\-si Ma\-te\-ma\-ti\-ca, la Pro\-ba\-bi\-li\-t\`{a} e le lo\-ro
Appli\-ca\-zio\-ni (GNAMPA) of the Isti\-tu\-to Na\-zio\-na\-le di Al\-ta Ma\-te\-ma\-ti\-ca (INdAM).
Progetto di Ricerca 2016: ``Problemi differenziali non lineari: esistenza, molteplicit\`{a} e propriet\`{a} qualitative delle soluzioni''. 
It is also partially supported by the project ``Existence and asymptotic behavior of solutions to systems of semilinear elliptic partial differential equations'' (T.1110.14) of the \textit{Fonds de la Recherche Fondamentale Collective}, Belgium.}
}
\author{{\bf\large Guglielmo Feltrin}
\vspace{1mm}\\
{\it\small D\'{e}partement de Math\'{e}matique, Universit\'{e} de Mons}\\
{\it\small Place du Parc 20}, {\it\small B-7000 Mons, Belgium}\\
{\it\small e-mail: guglielmo.feltrin@umons.ac.be}\vspace{1mm}
}

\date{}

\maketitle

\vspace{-2mm}

\begin{abstract}
\noindent
We study the positive subharmonic solutions to the second order nonlinear ordinary differential equation
\begin{equation*}
u'' + q(t)  g(u) = 0,
\end{equation*}
where $g(u)$ has superlinear growth both at zero and at infinity, and $q(t)$ is a $T$-periodic sign-changing weight.
Under the sharp mean value condition $\int_{0}^{T} q(t) ~\!dt < 0$, combining Mawhin's coincidence degree theory with the Poincar\'{e}-Birkhoff fixed point theorem, we prove that there exist positive subharmonic solutions of order $k$ for any large integer $k$. Moreover, when the negative part of $q(t)$ is sufficiently large, using a topological approach still based on coincidence degree theory, we obtain the existence of positive subharmonics of order $k$ for any integer $k\geq2$.
\end{abstract}

\section{Introduction}\label{section-1}

In this paper we deal with positive subharmonic solutions to nonlinear differential equations with indefinite weight and we gather some results recently obtained in \cite{BoFe-16,FeZa-pp2015,FeZa-15ade}.

\medskip

Our investigation is devoted to the second order differential equation
\begin{equation}\label{eq-int-f}
u'' + f(t,u)=0,
\end{equation}
especially when the nonlinear vector field is of the form
\begin{equation}\label{eq-1.2}
f(t,u) := q(t)g(u),
\end{equation}
thus covering the classical \textit{superlinear indefinite case},
namely $g(u) = |u|^{p-1}u$, with $p > 1$, and $q(t)$ a sign-changing coefficient.

The investigation of boundary value problems associated with equation \eqref{eq-int-f} when $f$ is \textit{superlinear at infinity} with respect to $s$, namely
\begin{equation*}
\lim_{s\to\pm\infty} \dfrac{f(t,s)}{s} = \infty,
\end{equation*}
is a topic which has been widely studied, employing various different approaches. We refer to the introduction in \cite{PaZa-00} and the references therein for an interesting historical presentation on the subject.

In the present paper we deal with the \textit{periodic indefinite problem} associated with \eqref{eq-int-f}, namely we suppose that $t\to f(t,s)$ is a $T$-periodic and sign-changing map.
In this framework, starting with the pioneering work \cite{Bu-76} by Butler, there are a lot of results for \textit{oscillatory} solutions.  In particular, they provide infinitely many periodic and subharmonic solutions  with a large number of zeros, as well as globally bounded solutions defined on the real line and exhibiting a complex behavior (see for instance \cite{CaDaPa-02,PaZa-04,TeVe-00}).
We remark that large sign-changing solutions to \eqref{eq-int-f} have a strong oscillatory behavior and, in the indefinite case, some solutions blow up. Accordingly, the study of sign-changing solutions of indefinite problems has two main feature: the absence of a priori bounds and the non-continuability of some solutions. For these reasons the analysis of sign-changing solutions is delicate and 
strong regularity assumptions on the nonlinearity are required, including that $f(t,s)$ is continuous, is locally Lipschitz in $t$ and of locally bounded variation in $s$, the set of values of $t$ for which $f(t,s)=0$ is an isolated set (cf.~\cite{Bu-76} and subsequent contributions). However, we stress that no growth condition at $s=0$ is required.

Our investigation is dedicated to the study of \textit{positive} solutions of the \textit{periodic indefinite problem} associated with \eqref{eq-int-f} and provide a double contribution in this context. On one hand, we are going to present a topological approach that allows to avoid all the above regularity condition on $f(t,s)$: a minimal set of assumptions on the nonlinearity will be required, but including the superlinear growth at zero (cf.~Remark~\ref{rem-2.1}). On the other hand, our paper is one of the fewer investigations on \textit{positive} periodic solutions to equations like $(\mathscr{E})$ (see \cite{BaBoVe-15,FeZa-pp2015}, dealing with existence, multiplicity and chaotic dynamics of positive solutions). This latter aspect places in the investigation on indefinite equations of the form
\begin{equation*}
-\Delta u = q(x) g(u), \quad u \in \Omega \subseteq \mathbb{R}^{N},
\end{equation*}
that arise in many models concerning population dynamics, differential geometry and mathematical physics, and for which only non-negative solutions make sense. Concerning indefinite problems, we mention the contributions \cite{AlTa-93,AmLG-98,BeCDNi-94,HeKa-80} and we refer to the introductions in \cite{Ac-09,Bo-16sur,Fe-16thesis,FeZa-pp2015,So-pp2016} for a more complete discussion and bibliography on the subject.

\medskip

We can now illustrate our results. Let ${\mathbb{R}}^{+}:=\mathopen{[}0,+\infty\mathclose{[}$ denote the set of non-negative real numbers and let
$g\colon {\mathbb{R}}^{+} \to {\mathbb{R}}^{+}$ be a continuously differentiable function such that
\begin{equation*}
g(0) = 0, \qquad g(s) > 0 \quad  \text{for } \; s > 0.
\leqno{(g_{*})}
\end{equation*}
Let $T>0$ and let $q\colon \mathbb{R} \to \mathbb{R}$ be a $T$-periodic locally integrable function.

In this survey we focus our attention on the second order ordinary differential equation
\begin{equation*}
u''+q(t)g(u)=0.
\leqno{(\mathscr{E})}
\end{equation*}
Our main goal is the investigation of positive subharmonic solutions to $(\mathscr{E})$ when $q(t)$ is a sign-changing function and $g(s)$ satisfies the following condition
\begin{equation*}
g'(0) = 0
\quad \text{ and } \quad
\lim_{s\to +\infty} \dfrac{g(s)}{s} = +\infty,
\leqno{(g_{s})}
\end{equation*}
namely when $g(s)$ has a superlinear growth at zero and at infinity, thus the classical case $g(s)=s^{p}$, with $p>1$, is covered.

In this paper, following a standard definition, we use the terminology \textit{subharmonic solution of order $k$} to $(\mathscr{E})$ (where $k\geq 2$ is an integer number) to indicate a $kT$-periodic solution to $(\mathscr{E})$ which is not $\ell T$-periodic
for any $\ell=1,\ldots,k-1$, in other words, $kT$ is the minimal period of $u(t)$ in the set of the integer multiples of $T$. It is worth noting that, assuming that $T$ is the minimal period of $q(t)$, from hypothesis $(g_{*})$ we derive that $kT$ is the minimal period of any positive subharmonic solution of order $k$ (cf.~the discussion in \cite[\S~4]{FeZa-pp2015}).
As a further remark, we underline that, if $u(t)$ is a positive subharmonic solution of order $k$ to $(\mathscr{E})$, then the $k-1$ time-translated functions $u(\cdot + \ell T)$, for $\ell =1, \ldots, k-1$,  are positive subharmonic solutions of order $k$ too. These solutions, though distinct, belong to the same periodicity class (in particular, they have to be considered equivalent when counting subharmonics).

From the above definition, it is clear that there are two main issues to face in the search of subharmonics to $(\mathscr{E})$: the existence of positive $kT$-periodic solutions to $(\mathscr{E})$ and the proof that $kT$ is the minimal period (in the sense described above) of some of these solutions (which is the most difficult point).

In this perspective, we now present two necessary conditions to the existence of positive $kT$-periodic solutions (where $k\geq1$ an integer number).
Indeed, if $u(t)$ is any positive $kT$-periodic solution to $(\mathscr{E})$, then integrating equation $(\mathscr{E})$ on $\mathopen{[}0,kT\mathclose{]}$ we obtain
\begin{equation*}
0 = - \int_{0}^{kT} u''(t) ~\!dt = \int_{0}^{kT} q(t)g(u(t))~\!dt.
\end{equation*}
Therefore, by $(g_{*})$, $q(t)$ has to change its sign (if not identically zero).
A second relation can be derived when $g'(s)>0$ for $s>0$, as in the case $g(s)=s^{p}$, with $p>1$. Precisely, dividing equation $(\mathscr{E})$ by $g(u(t))$ and integrating by parts, we find
\begin{equation*}
k\int_{0}^{T} q(t)~\!dt = \int_{0}^{kT} q(t)~\!dt = - \int_{0}^{kT} \biggl{(}\dfrac{u'(t)}{g(u(t))}\biggr{)}^{2} g'(u(t))~\!dt <0.
\end{equation*}
In the sequel we will show that this condition is also sufficient for the existence of subharmonics (cf.~Theorem~\ref{th-prel1}).

Since the main motivation for the present investigation is the superlinear equation $u''+q(t)u^{p}=0$ (with $p>1$), as a natural hypothesis we suppose that $q\colon \mathbb{R} \to \mathbb{R}$ is a $T$-periodic locally integrable \textit{sign-changing} function (i.e.~an \textit{indefinite} weight) satisfying the mean value condition
\begin{equation*}
\int_{0}^{T} q(t)~\!dt <0.
\leqno{(q_{\#})}
\end{equation*}
Additionally, we assume that in a time-interval of length $T$ there exists a finite number of closed pairwise disjoint subintervals where $q(t) \succ 0$ (i.e.~$q(t)\geq 0$ almost everywhere and $q\not\equiv 0$ on each interval), separated by closed intervals where $q(t) \prec 0$ (i.e.~$-q(t) \succ 0$).
More precisely, thanks to the periodicity of $q(t)$ and for ease of notation, we assume that
\begin{itemize}
\item[$(q_{*})$]
\textit{there exist $m \geq 1$ closed and pairwise disjoint intervals $I^{+}_{1},\ldots,I^{+}_{m}$
separated by $m$ closed intervals $I^{-}_{1},\ldots,I^{-}_{m}$ such that
\begin{equation*}
q(t)\succ 0 \; \text{ on } I^{+}_{i}, \qquad q(t)\prec 0 \; \text{ on } I^{-}_{i},
\end{equation*}
and, moreover,
\begin{equation*}
\bigcup_{i=1}^{m} I^{+}_{i} \, \cup \, \bigcup_{i=1}^{m} I^{-}_{i} = \mathopen{[}0,T\mathclose{]}.
\end{equation*}}
\end{itemize}

\medskip

In the manuscript we present two results of existence of infinitely many positive subharmonics.

The first one combines an application of the Poincar\'{e}-Birkhoff fixed point theorem, a smart trick used in \cite{BrHe-90} by Brown and Hess (that requires the strict convexity of $g(s)$) and coincidence degree theory. It states the following.

\begin{theorem}\label{th-main1}
Let $q \colon \mathbb{R} \to \mathbb{R}$ be a $T$-periodic locally integrable function satisfying $(q_{\#})$ and $(q_{*})$.
Let $g \in \mathcal{C}^{2}(\mathbb{R}^{+})$ satisfy $(g_{*})$, $(g_{s})$ and 
\begin{equation*}
g''(s) > 0 \quad \text{for } \; s>0.
\leqno{(g_{**})}
\end{equation*}
Then there exists a positive $T$-periodic solution $u^{*}(t)$ of equation $(\mathscr{E})$;
moreover, there exists $k^{*} \geq 1$ such that for any integer $k \geq k^{*}$ there exists an integer $m_{k} \geq 1$ such that, for any integer $j$ relatively prime with $k$ and such that
$1 \leq j \leq m_{k}$, equation $(\mathscr{E})$ has two positive subharmonic solutions $u_{k,j}^{(i)}(t)$ ($i=1,2$) of order $k$ (not belonging to the same periodicity class),
such that $u_{k,j}^{(i)}(t) - u^{*}(t)$ has exactly $2j$ zeros in the interval $\mathopen{[}0,kT\mathclose{[}$.
\end{theorem}

We underline that the subharmonic solutions obtained in Theorem~\ref{th-main1} oscillate around a positive $T$-periodic solution $u^{*}(t)$ of $(\mathscr{E})$: this property is crucial in order to obtain the ``minimality'' of the period.
We also stress that, taking $j=1$ in the statement of Theorem~\ref{th-main1}, we have the existence of two subharmonic solutions of order $k$ for any large integer $k$.

\medskip

The second result follows a line of research initiated by G\'{o}mez-Re\~{n}asco and L\'{o}pez-G\'{o}mez in \cite{GRLG-00}, where the authors asserted that the Dirichlet problem associated with $(\mathscr{E})$ has at least $2^{m}-1$ positive solutions when the negative part of the weight $q(t)$ is sufficiently large (and $m$ is the number of positive humps of $q(t)$ separated by negative ones, as in hypothesis $(q_{*})$).
According to a standard notation adopted in the contributions that followed from \cite{GRLG-00} (see, for instance, \cite{BoGoHa-05,BoDaPa-pp2016,FeZa-15jde,GaHaZa-03mod,GiGo-09jde}), it is convenient to introduce the parameter-dependent equation
\begin{equation*}
u'' + \bigl{(}a^{+}(t) - \mu a^{-}(t)\bigr{)} g(u) = 0,
\leqno{(\mathscr{E}_{\mu})}
\end{equation*}
with $\mu>0$ and
\begin{equation}\label{amu}
q(t) = a_{\mu}(t) := a^{+}(t) - \mu a^{-}(t), \quad t\in\mathbb{R},
\end{equation}
where $a^{+}(t)$ and $a^{-}(t)$ are, respectively, the positive and the negative part of a $T$-periodic locally integrable function $a \colon \mathbb{R} \to \mathbb{R}$.

In this setting, using a topological approach based on coincidence degree theory, we obtain the following.

\begin{theorem}\label{th-main2}
Let $q \colon \mathbb{R} \to \mathbb{R}$ be a $T$-periodic locally integrable function of the form \eqref{amu} satisfying $(q_{*})$.
Let $g \in \mathcal{C}^{1}(\mathbb{R}^{+})$ satisfy $(g_{*})$ and $(g_{s})$.
Then there exists $\mu^{*}>0$ such that, for all $\mu>\mu^{*}$ and for every integer $k\geq 2$,
equation $(\mathscr{E}_{\mu})$ has a subharmonic solution of order $k$.
\end{theorem}

Theorem~\ref{th-main2} ensures the existence of infinitely many subharmonics taking $\mu>0$ sufficiently large. We stress that in the statement we do not assume condition $(q_{\#})$, since it is implicitly satisfied taking $\mu>\mu^{*}$ with
\begin{equation}\label{eq-mud}
\mu^{*} \geq \mu^{\#} := \dfrac{\int_{0}^{T}a^{+}(t)~\!dt}{\int_{0}^{T}a^{-}(t)~\!dt}.
\end{equation}

\medskip

The plan of the paper is the following.
In Section~\ref{section-2} we illustrate two preliminary results concerning existence and multiplicity of positive $T$-periodic solutions to equations $(\mathscr{E})$ and $(\mathscr{E}_{\mu})$, respectively; the proofs of both results are based on coincidence degree theory and are provided in \cite{FeZa-pp2015,FeZa-15ade}.
In Section~\ref{section-3} we give the proof of Theorem~\ref{th-main1}; some remarks on different possible generalizations are also given.
Section~\ref{section-4} is devoted to Theorem~\ref{th-main2}: we prove the existence of infinitely many positive subharmonic solutions when the negative part of the weight is large enough and, in addition, we estimate the number of subharmonics of a given order. Finally, in Section~\ref{section-5} we compare the two main results and present some open questions.
We conclude this paper with Appendix~\ref{appendix-A} where we discuss some basic facts about the coincidence degree defined in open and possibly unbounded sets and we state some lemmas for the computation of the
degree, employed in Section~\ref{section-2}.

\section{Preliminary results: positive $T$-periodic solutions}\label{section-2}

In this section we present two theorems that ensure existence and, respectively, multiplicity of positive $T$-periodic solutions to $(\mathscr{E})$.
We take advantage of a topological approach introduced in \cite{FeZa-pp2015,FeZa-15ade} based on Mawhin's coincidence degree theory (cf.~\cite{GaMa-77,Ma-79,Ma-93}).

\medskip

The first result states the existence of a positive $T$-periodic solution when the mean value condition $(q_{\#})$ holds.
Therefore, this theorem gives an answer to a question raised by Butler in \cite{Bu-76}. In \cite{Bu-76} the author proved that equation
\begin{equation}
u''+q(t)|u|^{p-1}u=0, \quad p>1,
\end{equation}
has infinitely many $T$-periodic solutions and all these solutions oscillate (have arbitrarily large zeros), by assuming that $q(t)$ is a continuous $T$-periodic function
with only isolated zeros and such that
\begin{equation*}
\int_{0}^{T} q(t)~\!dt \geq0.
\end{equation*}
Moreover, underlining that condition $(q_{\#})$ implies the existence of non-oscillatory solutions, Butler raised the question whether there can exist positive \textit{periodic} solutions under hypothesis $(q_{\#})$. This is stated in the following theorem.

\begin{theorem}\label{th-prel1}
Let $q \colon \mathbb{R} \to \mathbb{R}$ be a $T$-periodic locally integrable function satisfying $(q_{\#})$ and $(q_{*})$.
Let $g \in \mathcal{C}^{1}(\mathbb{R}^{+})$ satisfy $(g_{*})$ and $(g_{s})$.
Then there exists at least a positive $T$-periodic solution of equation $({\mathscr{E}})$.
\end{theorem}

\begin{proof}
We give only a sketch of the proof, describing the main steps (which are developed in details in \cite{FeZa-15ade}).

\smallskip
\noindent
\textit{Step~1. Mawhin's coincidence degree setting. }
First of all, using a standard procedure, we define the $L^{1}$-Carath\'{e}odory function $f\colon\mathbb{R}^{2}\to\mathbb{R}$ as
\begin{equation*}
f(t,s) :=
\begin{cases}
\, -s, & \text{if } s \leq 0;\\
\, q(t)g(s), & \text{if } s\geq 0;
\end{cases}
\end{equation*}
and we deal with the $T$-periodic problem associated with
\begin{equation}\label{eq-f}
u'' + f(t,u) = 0.
\end{equation}
Via a standard maximum principle, one can prove that every $T$-periodic solution $u(t)$ of \eqref{eq-f} is non-negative and if $u(t)$ is a $T$-periodic solution of \eqref{eq-f} with $u\not\equiv0$, then $u(t)>0$ for all $t\in\mathbb{R}$.

Secondly, we write the $T$-periodic problem associated with \eqref{eq-f} as a \textit{coincidence equation}
\begin{equation}\label{coinc-eq}
Lu = Nu,\quad u\in \text{\rm dom}\,L.
\end{equation}
Taking into account that solving the $T$-periodic periodic problem associated with \eqref{eq-f} is equivalent to solving equation \eqref{eq-f} on $\mathopen{[}0,T\mathclose{]}$ together with the periodic boundary condition $u(0)=u(T)$ and $u'(0)=u'(T)$, we set $X:=\mathcal{C}(\mathopen{[}0,T\mathclose{]})$, the Banach space of continuous functions $u \colon \mathopen{[}0,T\mathclose{]} \to \mathbb{R}$
endowed with the $\sup$-norm $\|u\|_{\infty} := \max_{t\in \mathopen{[}0,T\mathclose{]}} |u(t)|$,
and $Z:=L^{1}(\mathopen{[}0,T\mathclose{]})$, the Banach space of integrable functions $v \colon \mathopen{[}0,T\mathclose{]} \to \mathbb{R}$ endowed with the norm $\|v\|_{L^{1}}:= \int_{0}^{T} |v(t)|~\!dt$.
Next, on $\text{\rm dom}\,L := \{u\in W^{2,1}(\mathopen{[}0,T\mathclose{]}) \colon u(0) = u(T), \, u'(0) = u'(T) \} \subseteq X$ we define the differential operator
\begin{equation*}
L \colon u \mapsto - u'',
\end{equation*}
which is a linear Fredholm map of index zero. Moreover, in order to enter the coincidence degree setting, we introduce the projectors
$P \colon X \to \ker L \cong {\mathbb{R}}$, $Q \colon Z \to \text{\rm coker}\,L \cong Z/\text{\rm Im}\,L \cong \mathbb{R}$,
the right inverse $K_{P} \colon \text{\rm Im}\,L \to \text{\rm dom}\,L \cap \ker P$ of $L$,
and the orientation-preserving isomorphism $J \colon \text{\rm coker}\,L \to \ker L$.
For the standard definition of these operators we refer to \cite[\S~2]{BoFeZa-16prse} and to \cite[\S~2]{FeZa-15ade}.
Finally, let $N \colon X \to Z$ be the Nemytskii operator induced by the nonlinear function $f(t,s)$, that is
\begin{equation*}
(N u)(t):= f(t,u(t)), \quad t\in \mathopen{[}0,T\mathclose{]}.
\end{equation*}

\smallskip

With this position, now we show how to reach the thesis using a topological approach based on Mawhin's coincidence degree. We refer to \cite{GaMa-77,Ma-79,Ma-93} for the classical definition and properties of the \textit{coincidence degree} $D_{L}(L-N,\Omega)$ \textit{of $L$ and $N$ in $\Omega$}, where $\Omega\subseteq X$ is an open and bounded set (cf.~also Appendix~\ref{appendix-A}).

\smallskip
\noindent
\textit{Step~2. Degree on a small ball. }
Since $g'(0)=0$, we can fix a (small) constant $r>0$ such that the following property holds.
\begin{itemize}
\item
If $\vartheta\in \mathopen{]}0,1\mathclose{]}$ and $u(t)$ is any non-negative $T$-periodic solution of
\begin{equation*}
u'' + \vartheta q(t) g(u) = 0, \\
\end{equation*}
then $\|u\|_{\infty}\neq r$.
\end{itemize}
Then, using condition $(q_{\#})$, by Lemma~\ref{lemma_Mawhin} we obtain that
\begin{equation}\label{eq-degr}
D_{L}(L-N,B(0,r)) = 1.
\end{equation}

\smallskip
\noindent
\textit{Step~3. Degree on a large ball. }
Since $g(s)/s\to+\infty$ as $s\to+\infty$, we can fix a (large) constant $R > 0$ (with $R>r$) such that the following property holds.
\begin{itemize}
\item
There exist a non-negative function $v\in L_{T}^{1}$ with $v\not\equiv 0$
and a constant $\nu_{0} > 0$, such that every non-negative $T$-periodic solution $u(t)$ of
\begin{equation}\label{eq-lem}
u'' + q(t) g(u) + \nu v(t) = 0,
\end{equation}
for $\nu \in \mathopen{[}0,\nu_{0}\mathclose{]}$, satisfies $\|u\|_{\infty} \neq R$.
Moreover, there are no $T$-periodic solutions $u(t)$ of \eqref{eq-lem} for $\nu = \nu_{0}$ with $0 \leq u(t) \leq R$,
for all $t\in \mathbb{R}$.
\end{itemize}
Then, by Lemma~\ref{lem-abs-deg0} we derive that
\begin{equation}\label{eq-degR}
D_{L}(L-N,B(0,R)) = 0.
\end{equation}

\smallskip
\noindent
\textit{Step~4. Conclusion. }
From \eqref{eq-degr} and \eqref{eq-degR}, using the additivity property of Mawhin's coincidence degree, we derive that
\begin{equation*}
D_{L}( L-N, B(0,R) \setminus \overline{B(0,r)})=-1.
\end{equation*}
Then, by the existence property of the degree, there exists at least a nontrivial solution $u^{*}(t)$ of \eqref{coinc-eq} with $r<\|u^{*}\|_\infty < R$.
Via a standard maximum principle, we conclude that $u^{*}(t)$ is a positive $T$-periodic solution of \eqref{eq-f} and thus of $(\mathscr{E})$.
The theorem follows.
\end{proof}

\medskip

Now, we continue the investigation on Butler's open problem, by providing multiple positive $T$-periodic solutions to $(\mathscr{E})$ (depending on the nodal properties of the weight $q(t)$) when the negative part of $q(t)$ is sufficiently large. Accordingly, for this second part of the section, we deal with the parameter-dependent equation 
\begin{equation*}
u'' + \bigl{(}a^{+}(t) - \mu a^{-}(t)\bigr{)} g(u) = 0,
\leqno{(\mathscr{E}_{\mu})}
\end{equation*}
where $\mu>0$, and we prove our result of multiplicity when $\mu$ is large enough.
In the sequel, when dealing with the function $a(t)$, we denote with $(a_{*})$ the hypothesis about the existence of $m$ intervals where $a(t)\succ0$ separated by $m$ intervals where $a(t)\prec0$ (in $\mathopen{[}0,T\mathclose{]}$), which is the analogous of condition $(q_{*})$ referring to $q(t)$.

The following theorem holds (see also Figure~\ref{fig-00}).

\begin{theorem}\label{th-prel2}
Let $a \colon \mathbb{R} \to \mathbb{R}$ be a $T$-periodic locally integrable function satisfying $(a_{*})$.
Let $g \in \mathcal{C}^{1}(\mathbb{R}^{+})$ satisfy $(g_{*})$ and $(g_{s})$.
Then there exists $\mu^{*}>0$ such that for all $\mu>\mu^{*}$ equation $(\mathscr{E}_{\mu})$ has at least $2^{m}-1$ positive $T$-periodic solutions.
\end{theorem}

\begin{proof}
We give only a sketch of the proof, describing the main steps (which are developed in details in \cite{FeZa-pp2015}).

\smallskip
\noindent
\textit{Step~1. Notation and Mawhin's coincidence degree setting. }
For technical reasons, without loss of generality and consistently with assumption $(a_{*})$, we can select the endpoints of each interval of positivity $I^{+}_{i} = \mathopen{[}\sigma,\tau\mathclose{]}$ in such a manner that $a(t)\not\equiv 0$ on all left neighborhoods of $\sigma$ and on all right neighborhoods of $\tau$. In \cite{BoFeZa-17tams} it is showed, in a different context, how this additional technical hypothesis can be avoided.

First of all, we fix the constants $r$ and $R$, with $0<r<R$, as in \textit{Step~2} and \textit{Step~3} of the proof of Theorem~\ref{th-prel1}.
Next, for every subset of indices $\mathcal{I}\subseteq\{1,\ldots,m\}$ (possibly empty), we define two open unbounded sets
\begin{equation*}
\begin{aligned}
\Omega^{\mathcal{I}}:=
\biggl\{\,u\in\mathcal{C}(\mathopen{[}0,T\mathclose{]})\colon
   & \max_{t\in I^{+}_{i}}|u(t)|<R, \, i\in\mathcal{I};                             &
\\ & \max_{t\in I^{+}_{i}}|u(t)|<r, \, i\in\{1,\ldots,m\}\setminus\mathcal{I}       & \biggr\}
\end{aligned}
\end{equation*}
and
\begin{equation}\label{eq-lambda}
\begin{aligned}
\Lambda^{\mathcal{I}}:=
\biggl\{\,u\in\mathcal{C}(\mathopen{[}0,T\mathclose{]})\colon
   & r < \max_{t\in I^{+}_{i}}|u(t)|<R, \, i\in\mathcal{I};                           &
\\ & \max_{t\in I^{+}_{i}}|u(t)|<r, \, i\in\{1,\ldots,m\}\setminus\mathcal{I}     & \biggr\}.
\end{aligned}
\end{equation}

Now we enter the setting of Mawhin's coincidence degree in the same manner as we have done in \textit{Step~1} of the proof of Theorem~\ref{th-prel1}. Our goal is to compute the degree on the sets $\Omega^{\mathcal{I}}$ and $\Lambda^{\mathcal{I}}$, in particular we are going to prove that
\begin{equation*}
D_{L}(L-N,\Lambda^{\mathcal{I}})\neq0, \quad \forall \, \mathcal{I}\subseteq\{1,\ldots,m\}.
\end{equation*}

Since the sets $\Omega^{\mathcal{I}}$ and $\Lambda^{\mathcal{I}}$ are open and \textit{unbounded}, we use the more general version of the coincidence degree for locally compact operators on open and possibly unbounded
sets (see Appendix~\ref{appendix-A} and \cite[Appendix~B]{Fe-16thesis}).

\smallskip
\noindent
\textit{Step~2. Degree on $\Omega^{\mathcal{I}}$. }
For any subset of indices $\mathcal{I}\subseteq\{1,\ldots,m\}$ we compute $D_{L}(L-N,\Omega^{\mathcal{I}})$.
First of all, we consider the case $\mathcal{I}=\emptyset$. As a consequence of a convexity argument, the excision property of the degree and formula \eqref{eq-degr}, we obtain that
\begin{equation}\label{deg-emptyset}
D_{L}(L-N,\Omega^{\emptyset}) = D_{L}(L-N,B(0,r))=1.
\end{equation}
Secondly, let us consider a subset $\mathcal{I}\neq\emptyset$. Via a homotopic argument (by Lemma~\ref{lem-deg0-deFigueiredo}), we can prove the existence of $\mu^{*}\geq\mu^{\#}>0$ (where $\mu^{\#}$ is the constant defined in \eqref{eq-mud}) such that for all $\mu>\mu^{*}$ the degree $D_{L}(L-N,\Omega^{\mathcal{I}})$ is well-defined and the following formula holds
\begin{equation}\label{deg-0}
D_{L}(L-N,\Omega^{\mathcal{I}})=0, \quad \forall \, \emptyset\neq\mathcal{I}\subseteq\{1,\ldots,m\}.
\end{equation}

\smallskip
\noindent
\textit{Step~3. Degree on $\Lambda^{\mathcal{I}}$. }
For any subset of indices $\mathcal{I}\subseteq\{1,\ldots,m\}$ we compute $D_{L}(L-N,\Lambda^{\mathcal{I}})$.
Via a purely combinatorial argument (done by induction on the cardinality $\#\mathcal{I}$ of the set $\mathcal{I}$), from \eqref{deg-emptyset} and \eqref{deg-0}, we deduce that for all $\mu>\mu^{*}$ the degree $D_{L}(L-N,\Lambda^{\mathcal{I}})$ is well-defined and the following formula holds
\begin{equation*}
D_{L}(L-N,\Lambda^{\mathcal{I}})=(-1)^{\#\mathcal{I}}, \quad \forall \, \mathcal{I}\subseteq\{1,\ldots,m\}.
\end{equation*}

\smallskip
\noindent
\textit{Step~4. Conclusion. }
Preliminarily, we underline that $0\notin\Lambda^{\mathcal{I}}$ for all $\emptyset\neq\mathcal{I}\subseteq\{1,\ldots,m\}$
and the sets $\Lambda^{\mathcal{I}}$ are pairwise disjoint. Since the number of nonempty subsets of a set with $m$ elements is $2^{m}-1$, there are 
$2^{m}-1$ sets $\Lambda^{\mathcal{I}}$ not containing the null function.
Since, from \textit{Step~3}, in particular we have that
\begin{equation*}
D_{L}(L-N,\Lambda^{\mathcal{I}})\neq0, \quad \forall \, \mathcal{I}\subseteq\{1,\ldots,m\},
\end{equation*}
we conclude that there exist at least $2^{m}-1$ nontrivial solution of \eqref{coinc-eq}.
Via a standard maximum principle, we conclude that these nontrivial functions are positive $T$-periodic solutions of \eqref{eq-f} and thus of $(\mathscr{E}_{\mu})$.
The theorem follows.
\end{proof}

\begin{figure}[h!]
\centering
\begin{tikzpicture} [scale=1]
\node at (-3,0) {\includegraphics[width=0.45\textwidth]{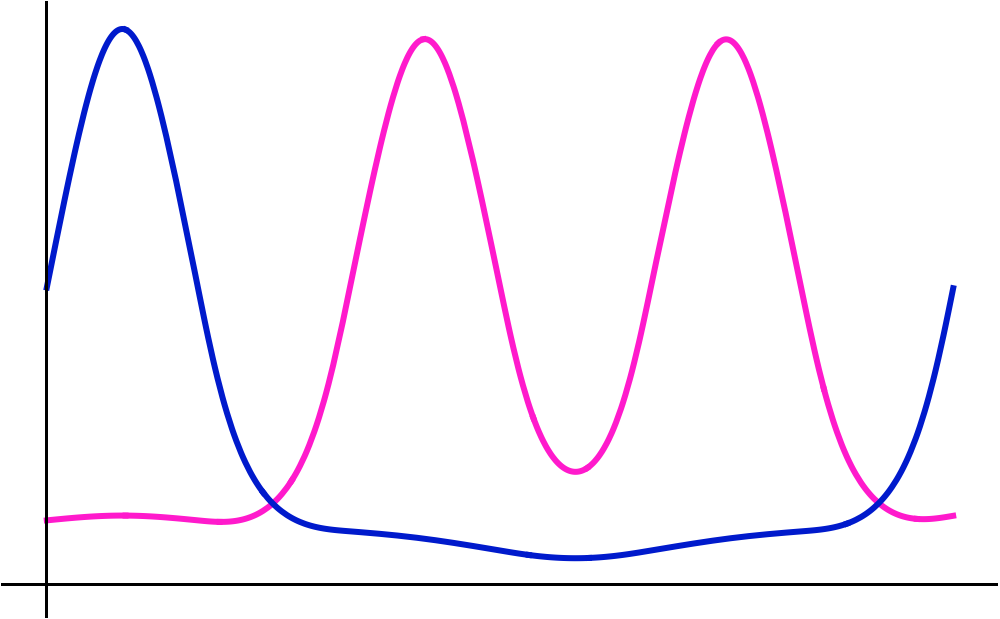}};
\node at (3,0) {\includegraphics[width=0.45\textwidth]{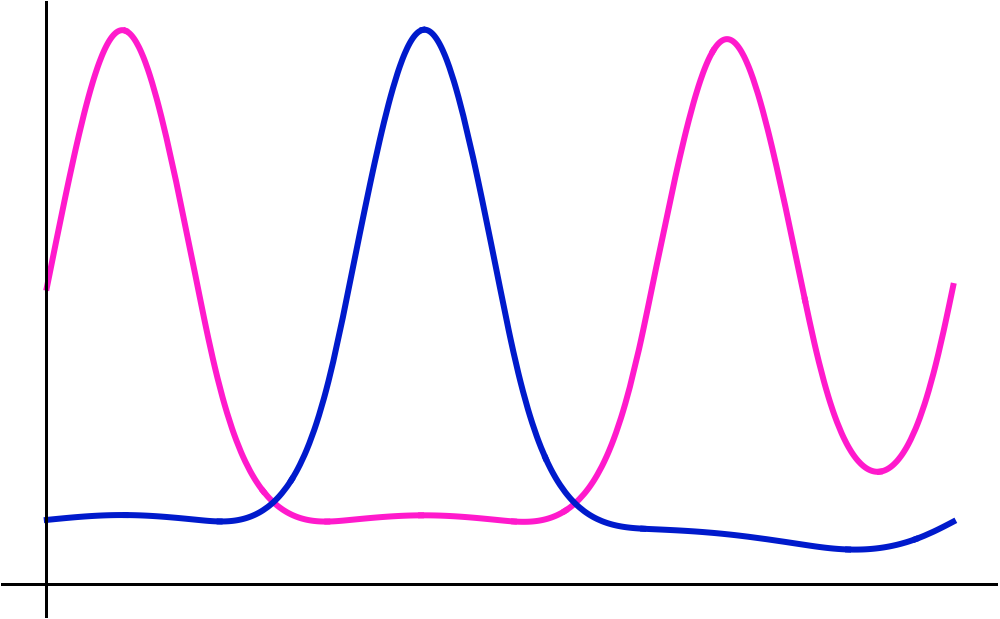}};
\node at (-3,-4) {\includegraphics[width=0.45\textwidth]{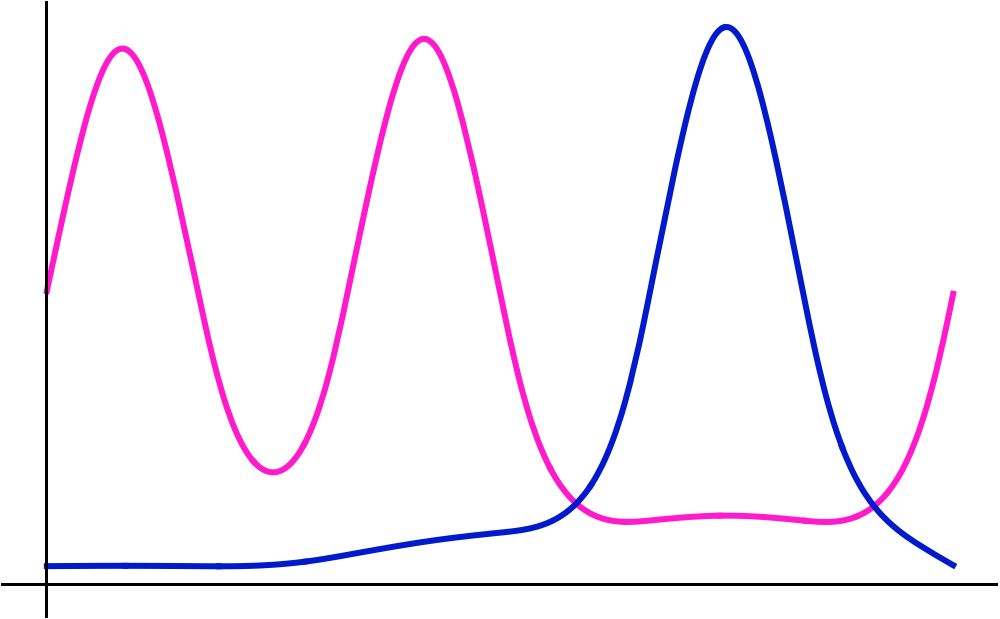}};
\node at (3,-4) {\includegraphics[width=0.45\textwidth]{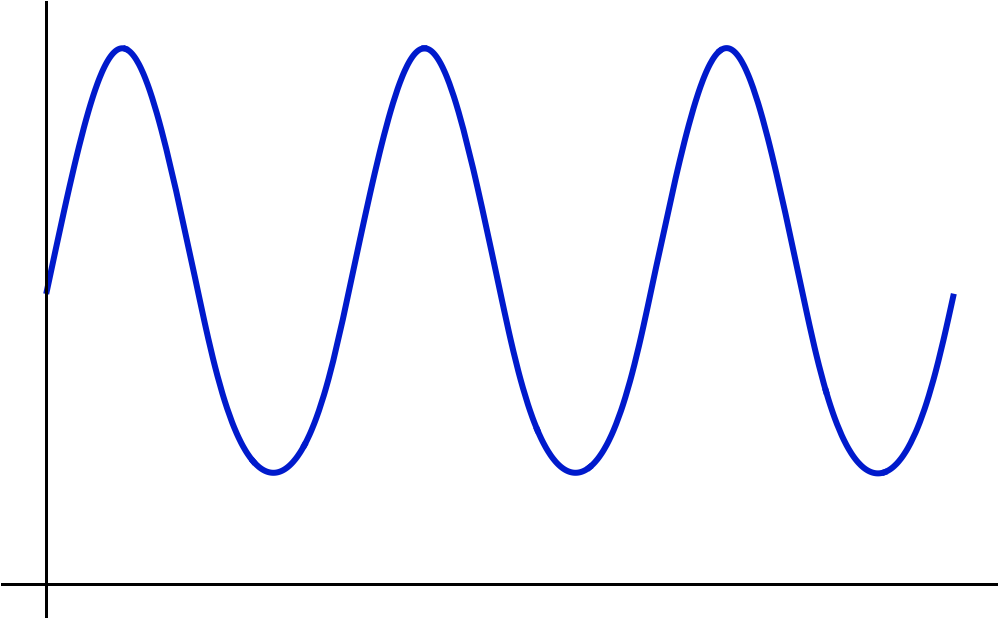}};
\end{tikzpicture}
\caption{The figure shows an example of multiple positive solutions for the $T$-periodic boundary value problem associated with $(\mathscr{E}_{\mu})$.
For this numerical simulation we have chosen $a(t) = \sin(6\pi t)$ for $t\in\mathopen{[}0,1\mathclose{]}$, $\mu = 10$ and $g(s) = \max\{0,400\, s\arctan|s|\}$. Notice that the weight function $a(t)$ has $3$ positive humps.
We show the graphs of the $7$ positive $T$-periodic solutions of $(\mathscr{E}_{\mu})$. We stress that $g(s)/s\not\to+\infty$ as $s\to+\infty$ (contrary to what is assumed in $(g_{s})$), indeed Theorem~\ref{th-prel2} is also valid assuming only that $g(s)/s$ is sufficiently large as $s\to+\infty$, as observed in Remark~\ref{rem-2.1}.
}
\label{fig-00}
\end{figure}

\begin{remark}\label{rem-2.1}
We underline that one can employ a similar topological approach to obtain existence and multiplicity results (analogous to Theorem~\ref{th-prel1} and Theorem~\ref{th-prel2}, respectively) dealing with an equation involving also a friction term as
\begin{equation*}
u'' + cu' + \bigl{(}a^{+}(t) - \mu a^{-}(t)\bigr{)} g(u) = 0,
\leqno{(\mathscr{E}_{\mu,c})}
\end{equation*}
where $c\in\mathbb{R}$ is an arbitrary constant.
In fact, in \cite{FeZa-pp2015} it is showed how suitable monotonicity properties of the map $t\mapsto e^{ct}u'(t)$ replace the convexity/concavity
of the solutions of $(\mathscr{E}_{\mu})$. This is the crucial point that allows to adapt the topological approach described above in this more general setting.
However, even if both main results about subharmonic solutions, that will be proved in the following sections, are based on these existence/multiplicity theorems, only one is valid in the non-variational context (cf.~also Remark~\ref{rem-4.1}) and thus we prefer to skip the details of these general results.

Furthermore, we observe that Theorem~\ref{th-prel1} and Theorem~\ref{th-prel2} can be proved under less restrictive regularity hypotheses on $g(s)$.
Concerning Theorem~\ref{th-prel1}, the original version in \cite{FeZa-15ade} was stated and proved assuming $g(s)$ only continuous on $\mathbb{R}^{+}$ and continuously differentiable on a right neighborhood of zero (cf.~\cite[Theorem~3.2]{FeZa-15ade}), or assuming $g(s)$ continuous on $\mathbb{R}^{+}$ and regularly oscillating at zero (cf.~\cite[Theorem~3.1]{FeZa-15ade} and the references listed at the end of \cite[\S~1]{FeZa-15ade}).
In fact, in Theorem~\ref{th-prel1} we assume $g(s)$ of class $\mathcal{C}^{1}$ but we stress that in the proof only the continuity of $g(s)$ and its continuous differentiability near zero is used (see the details of the proof in \cite{FeZa-15ade}).

In this perspective, also Theorem~\ref{th-prel2} is valid under weaker hypotheses and, in fact, in  \cite{FeZa-pp2015} the authors proved it assuming only that $g\in\mathcal{C}(\mathbb{R}^{+})$ (with $\lim_{s\to0^{+}}g(s)/s\to0$ instead of $g'(0)=0$) and thus are able to avoid all the additional regularity conditions on $g(s)$ (by taking $\mu$ large enough).

It is worth noting that, dealing with equation $(\mathscr{E}_{\mu})$, the more general version of Theorem~\ref{th-prel1} in \cite{FeZa-15ade} ensures the existence of a positive $T$-periodic solution assuming $\mu>\mu^{\#}$ (with $\mu^{\#}$ defined in \eqref{eq-mud}), the continuity of $g(s)$ and some regularity conditions near $s=0$, while in the more general version of Theorem~\ref{th-prel2} in \cite{FeZa-pp2015} we obtain multiplicity of positive $T$-periodic solutions assuming $\mu$ large (that is $\mu>\mu^{*}\geq\mu^{\#}$) and the continuity of $g(s)$. We stress that in the second result in order to avoid additional regularity assumptions on the nonlinearity $g(s)$, we need a stronger hypothesis on the weight (namely take $\mu$ large) and thus a worse estimate on $\mu^{*}$.

Finally, we remark that in both theorems the growth condition at infinity can be weakened by dealing with nonlinearities $g(s)$ satisfying the less restrictive assumption at infinity
\begin{equation*}
\liminf_{s\to +\infty} \dfrac{g(s)}{s} > \max_{i=1,\ldots,m}\lambda_{1}^{i},
\end{equation*}
where $\lambda_{1}^{i}$ ($i=1,\ldots,m$) is the first eigenvalue of the eigenvalue problem on $I^{+}_{i}$
\begin{equation*}
\varphi'' + \lambda \, q(t) \varphi = 0, \quad \varphi|_{\partial I^{+}_{i}} = 0,
\end{equation*}
(cf.~hypothesis $(q_{*})$). We refer to \cite{FeZa-pp2015,FeZa-15ade} for the details (see also the example described in Figure~\ref{fig-00}).
However, we avoid investigations in this direction since our main result presented in Section~\ref{section-3} needs that $g(s)$ is of class $\mathcal{C}^{2}$ and satisfying $(g_{s})$.
$\hfill\lhd$
\end{remark}

\section{First result: symplectic approach}\label{section-3}

In this section we present a symplectic approach, based on the Poincar\'{e}-Birkhoff fixed point theorem, which allows to obtain infinitely many subharmonics to
\begin{equation*}
u'' + q(t)  g(u) = 0,
\leqno{(\mathscr{E})}
\end{equation*}
as stated in Theorem~\ref{th-main1}. We refer to \cite{BoFe-16} for the missing details.

\medskip

As a first step, analogously as in \textit{Step~3} of the proof of Theorem~\ref{th-prel1}, we prove the existence of $R>0$ (large enough) such that 
for every $\nu\geq0$ and for every integer $k\geq 1$ any $kT$-periodic solution $u(t)$ to
\begin{equation*}
u'' + q(t)g(u) + \nu \mathbbm{1}_{\bigcup_{i=1}^{m} I^{+}_{i}}(t) = 0
\end{equation*}
satisfies $\,\max_{t\in \mathbb{R}} u(t) <R$ (where $\mathbbm{1}_{J}$ denotes the indicator function of the interval $J$).
For simplicity, we can choose $R>0$ satisfying both the above property and the one in \textit{Step~3} of the proof of Theorem~\ref{th-prel1}.

As a second step, we introduce the extension $h \colon \mathbb{R}^{2} \to \mathbb{R}$ defined as
\begin{equation*}
h(t,s):=
\begin{cases}
\, 0, & \text{if } s\in \mathopen{]}-\infty, 0\mathclose{[}; \\
\, q(t)g(s), & \text{if } s\in \mathopen{[}0,R\mathclose{]}; \\
\, q(t)g(R) + q(t)g'(R)(s-R), & \text{if } s\in \mathopen{]}R,+\infty\mathclose{[};
\end{cases}
\end{equation*}
which is $T$-periodic in the first variable.
Then, we consider the differential equation
\begin{equation}\label{eqsub}
u'' + h(t,u) = 0.
\end{equation}
We observe that the global continuability for the solutions to \eqref{eqsub} is guaranteed, since the map $s\mapsto h(t,s)$ has linear growth at infinity.

Finally, we introduce the following notation. For any $q \in L^{1}(\mathopen{[}0,T\mathclose{]})$, we denote by $\lambda_{0}(q)$ the principal eigenvalue of the linear problem
\begin{equation}\label{eqhill}
v'' + \bigl{(}\lambda + q(t) \bigr{)} v = 0,
\end{equation}
with $T$-periodic boundary conditions. 

\medskip

The proof of Theorem~\ref{th-prel1} is grounded on the following result (see \cite{BoZa-13}).

\begin{proposition}\label{propsub}
Suppose that:
\begin{itemize}
\item[$(i)$] there exists a $T$-periodic solution $u^{*}(t)$ of \eqref{eqsub} satisfying
\begin{equation}\label{hpmorse}
\lambda_{0}(\partial_u h(t,u^{*}(t))) < 0;
\end{equation}
\item[$(ii)$] there exists a $T$-periodic function $\alpha \in W^{2,1}_{\textnormal{loc}}(\mathbb{R})$ satisfying
\begin{equation*}
\alpha''(t) + h(t,\alpha(t)) \geq 0, \quad \text{a.e. } t \in \mathbb{R},
\end{equation*}
and
\begin{equation*}
\alpha(t) < u^{*}(t), \quad \forall \, t \in \mathbb{R}.
\end{equation*}
\end{itemize}
Then there exists $k^{*} \geq 1$ such that for any integer $k \geq k^{*}$ there exists an integer $m_{k} \geq 1$ such that, for any integer $j$ relatively prime with $k$ and such that
$1 \leq j \leq m_{k}$, equation \eqref{eqsub} has two subharmonic solutions $u_{k,j}^{(1)}(t)$,  $u_{k,j}^{(2)}(t)$ of order $k$ (not belonging to the same periodicity class),
such that, for $i=1,2$, $u_{k,j}^{(i)}(t) - u^{*}(t)$ has exactly $2j$ zeros in the interval $\mathopen{[}0,kT\mathclose{[}$ and
\begin{equation*}
\alpha(t) \leq u_{k,j}^{(i)}(t), \quad \forall \, t \in \mathbb{R}.
\end{equation*}
\end{proposition}

The proof of Proposition~\ref{propsub} consists of an application of the Poincar\'{e}-Birkhoff fixed point theorem in $\mathbb{R}^{2}$.
The crucial point is to prove that there exists a $k$-th iterate of the Poincar\'{e} map associated with \eqref{eqhill} (which is an area-preserving homeomorphism) that satisfies the ``twist condition'' for an annular domain in the phase-plane. The main idea is to check that ``smaller'' solutions (departing sufficiently near the origin) make more than one revolution (around the origin) in $\mathopen{[}0,kT\mathclose{]}$ and ``larger'' solutions rotate very slowly and thus are unable to do a complete turn (around the origin).
We highlight that the ``minimality'' of the period of the $kT$-periodic solutions is a direct consequence of this powerful fixed point theorem.

\medskip

Furthermore, we introduce this crucial result about the computation of $\lambda_{0}(q(t)g'(u(t)) )$ for every positive $T$-periodic solution $u(t)$ of $(\mathscr{E})$ (independently of its existence).
The proof is based on the strict convexity assumption $(g_{**})$ and takes advantage of an algebraic trick adopted in \cite{BrHe-90}.

\begin{lemma}\label{lem-morse}
Let $q\colon \mathbb{R} \to \mathbb{R}$ be a $T$-periodic locally integrable function.
Let $g \in \mathcal{C}^{2}(\mathbb{R}^{+})$ satisfy $(g_{*})$ and $(g_{**})$.
If $u(t)$ is a positive $T$-periodic solution of $(\mathscr{E})$, then
\begin{equation*}
\lambda_{0}\bigl{(}q(t)g'(u(t)) \bigr{)}<0.
\end{equation*}
\end{lemma}

\medskip

As a third step, we are going to check that conditions $(i)$ and $(ii)$ of Proposition~\ref{propsub} are satisfied.

Theorem~\ref{th-prel1} ensures the existence of a positive $T$-periodic function $u^{*}(t)$ of $(\mathscr{E})$ with $r < \|u^{*}\|_{\infty} < R$ and, by Lemma~\ref{lem-morse}, $u^{*}(t)$ satisfies
\begin{equation*}
\lambda_{0}\bigl{(}q(t)g'(u^{*}(t)) \bigr{)} < 0.
\end{equation*}
Since $h(t,s) = q(t)g(s)$ for $0 \leq s \leq R$, we deduce that $u^{*}(t)$ is a (positive) $T$-periodic solution of \eqref{eqsub} satisfying \eqref{hpmorse}.
Hence, we have verified condition $(i)$ of Proposition~\ref{propsub}.

Next, taking $\alpha(t) \equiv 0$, we immediately have condition $(ii)$ verified, due to the fact that $\alpha(t)$ is a trivial solution of \eqref{eqsub} and $0 = \alpha(t) < u^{*}(t)$ for every $t \in \mathbb{R}$.

\medskip

As a final step, we apply Proposition~\ref{propsub}. Clearly, due to the uniqueness for the Cauchy problems, we have that $u_{k,j}^{(i)}(t) > 0$ for any $t \in \mathbb{R}$.
Moreover, from the choice of $R>0$ we obtain that $u_{k,j}^{(i)}(t) < R$ for any $t \in \mathbb{R}$. 
Hence, the proof of Theorem~\ref{th-main1} is concluded.
\qed

\begin{remark}\label{rem-3.1}
In \cite{BoFe-16}, a more general version of Theorem~\ref{th-main1} was proved and, as a consequence, the authors can also deal with nonlinearities $g(s)$ with a singularity, or nonlinearities $g(s)$ only with superlinear growth condition at zero (in this latter case, obtaining subharmonics to the parameter-dependent equation $u''+\lambda q(t)g(u)=0$ for $\lambda>0$ large enough).
$\hfill\lhd$
\end{remark}

\section{Second result: topological approach}\label{section-4}

In this section we present a topological approach, based on Mawhin's coincidence degree theory, which allows to obtain infinitely many subharmonics to
\begin{equation*}
u'' + \bigl{(}a^{+}(t) - \mu a^{-}(t)\bigr{)} g(u) = 0,
\leqno{(\mathscr{E}_{\mu})}
\end{equation*}
as stated in Theorem~\ref{th-main2}. We refer to \cite{FeZa-pp2015} for the missing details and other discussion on this issue.

\medskip

Let $a \colon \mathbb{R} \to \mathbb{R}$ be a $T$-periodic locally integrable function satisfying condition $(a_{*})$ (on the interval $\mathopen{[}0,T\mathclose{]}$) with $m$ intervals of positivity $I^{+}_{i}$ separated by $m$ intervals of negativity $I^{-}_{i}$.
In the following, $a(t)$ is treated as a $kT$-periodic function, for an integer $k\geq 2$.
Accordingly, defining the intervals
\begin{equation*}
J^{\pm}_{i,\ell}:= I^{\pm}_{i} + \ell\, T, \quad \text{ for } \; i=1,\ldots,m \; \text{ and } \; \ell \in \mathbb{Z},
\end{equation*}
in the periodicity interval $\mathopen{[}0,kT\mathclose{]}$ the weight $a(t)$ shows $km$ positive humps $J^{+}_{i,\ell}$ separated by $km$ negative ones $J^{-}_{i,\ell}$
(for $i=1,\ldots, m$ and $\ell= 0,\ldots, k-1$).

\medskip

As a first step, we apply Theorem~\ref{th-prel2} to the interval $\mathopen{[}0,kT\mathclose{]}$ and we obtain the following statement (see also Figure~\ref{fig-01}).

\begin{theorem}\label{th-sub}
Let $a \colon \mathbb{R} \to \mathbb{R}$ be a $T$-periodic locally integrable function satisfying $(a_{*})$.
Let $g \in \mathcal{C}^{1}(\mathbb{R}^{+})$ satisfy $(g_{*})$ and $(g_{s})$.
Then there exists $\mu^{*} > 0$ such that, for all $\mu > \mu^{*}$ and for every integer $k\geq2$,
equation $(\mathscr{E}_{\mu})$ has at least $2^{km}-1$ positive $kT$-periodic solutions.

More precisely, there exist two constants $r,R$ with $0 < r < R$ and $\mu^{*}>0$ such that, for any $\mu > \mu^{*}$ and for any integer $k\geq2$, 
given any $km$-tuple $\mathcal{S} = (s_{1},\ldots,s_{km}) \in \{0,1\}^{km}$, with $\mathcal{S} \neq (0,\ldots,0)$,
there exists at least a positive $kT$-periodic solution $u(t)$ of $(\mathscr{E}_{\mu})$ such that
\begin{itemize}
\item $\max_{t \in J^{+}_{i,\ell}} u(t) < r$, if $s_{j} = 0$ for $j= i + \ell m$;
\item $r < \max_{t \in J^{+}_{i,\ell}} u(t) < R$, if $s_{j} = 1$ for $j= i + \ell m$.
\end{itemize}
\end{theorem}

The crucial point of the proof is to observe that the constant $\mu^{*}$ does not depend on the integer $k\geq2$. Indeed, when applying Theorem~\ref{th-prel2} on the interval $\mathopen{[}0,kT\mathclose{]}$ we can choose the constant $\mu^{*}$ depending only on the local properties of $a(t)$ (in the interval $\mathopen{[}0,T\mathclose{]}$).
The thesis follows from Theorem~\ref{th-prel2} and from the definition \eqref{eq-lambda} of the sets $\Lambda^{\mathcal{I}}$ (with respect to the interval $\mathopen{[}0,kT\mathclose{]}$ and to the subintervals $J^{+}_{i,\ell}$, for $i=1,\ldots, m$ and $\ell= 0,\ldots, k-1$).

\begin{figure}[h!]
\centering
\begin{tikzpicture} [scale=1]
\node at (0,2.8) {\includegraphics[width=0.813\textwidth]{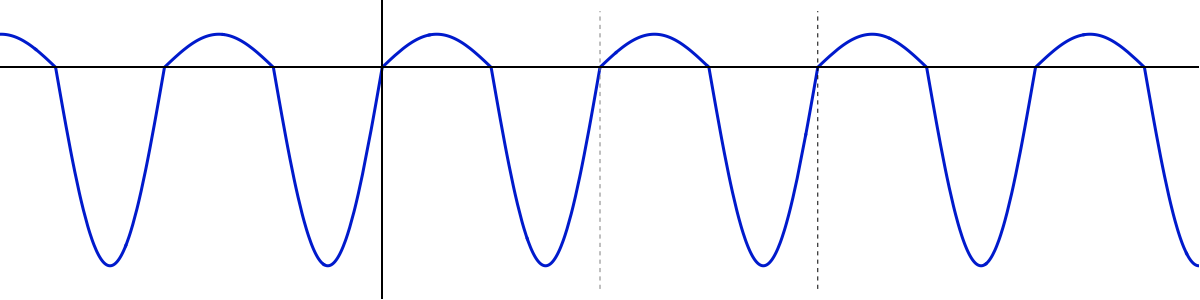}};
\node at (-3.8,4) {$a_{\mu}(t)$};
\node at (0.2,3.23) {$T$};
\node at (2.07,3.23) {$2T$};
\node at (-4,0) {\includegraphics[width=0.31\textwidth]{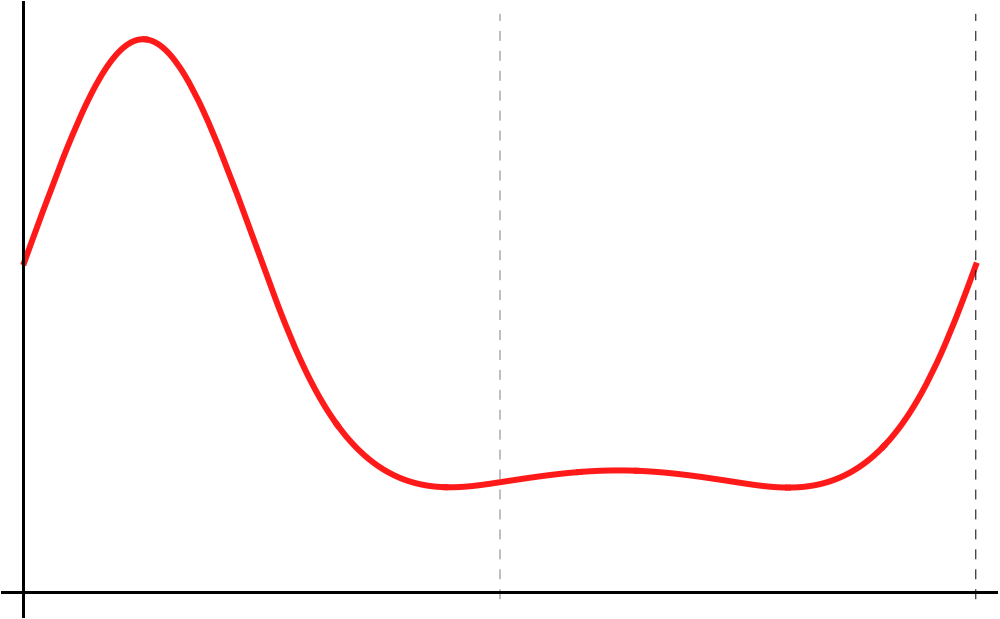}};
\node at (0,0) {\includegraphics[width=0.31\textwidth]{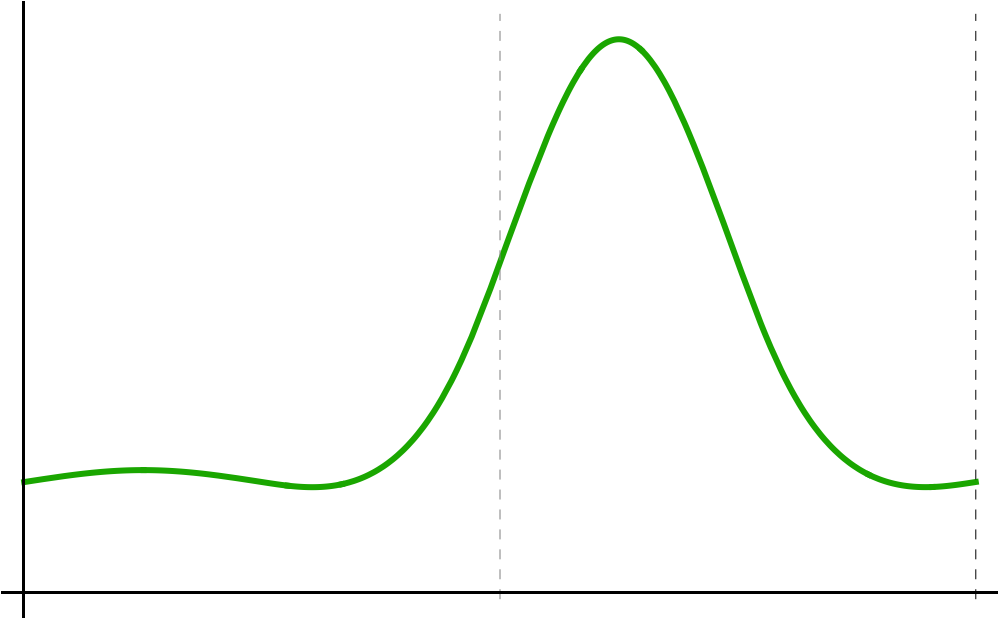}};
\node at (4,0) {\includegraphics[width=0.31\textwidth]{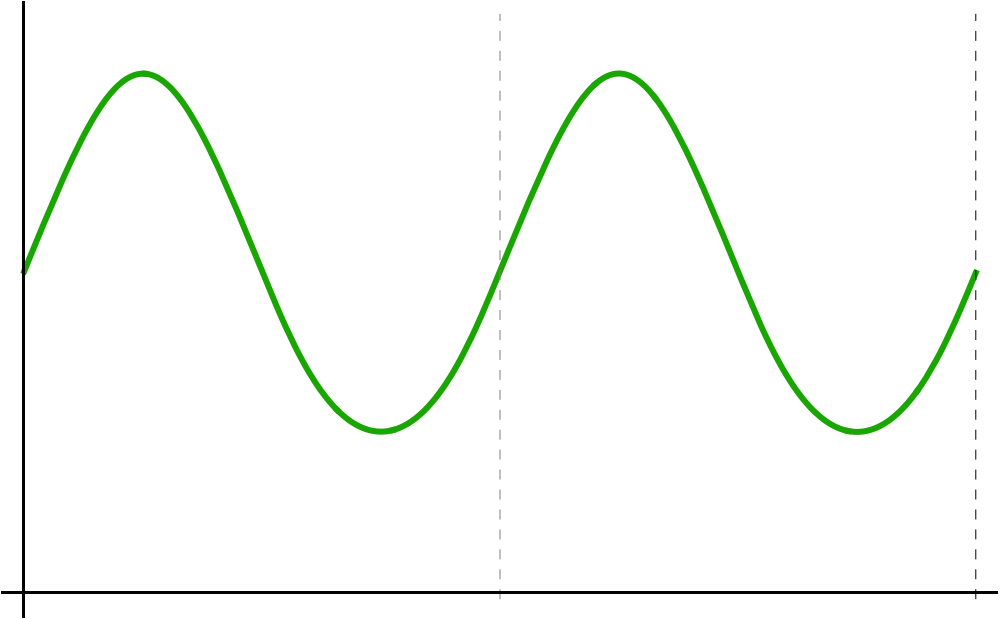}};
\node at (-4,-1.5) {$(1,0)$};
\node at (0,-1.5) {$(0,1)$};
\node at (4,-1.5) {$(1,1)$};
\end{tikzpicture}
\caption{Using a numerical simulation we have studied the $2$-periodic solutions of equation $(\mathscr{E}_{\mu})$ where $a(t) = \sin(t)$, $T=2\pi$, $\mu = 6$ and
$g(s) = 100(s^{2}+s^{3})$. In the upper part, the figure shows the graph of the weight $q(t)=a_{\mu}(t)$; while in the lower part, it shows the graphs of three $2T$-periodic positive solutions, in accordance with Theorem~\ref{th-sub}.
The three solutions are in correspondence with the three strings $(1,0)$, $(0,1)$ and $(1,1)$, respectively, as stated in Theorem~\ref{th-sub}.
The first two solutions are subharmonic solutions of order $2$ and the third one is a $1$-periodic solution.
As subharmonic solutions of order $2$, we consider only the first one, since the second one is a translation by $1$ of the first solution.
}
\label{fig-01}
\end{figure}
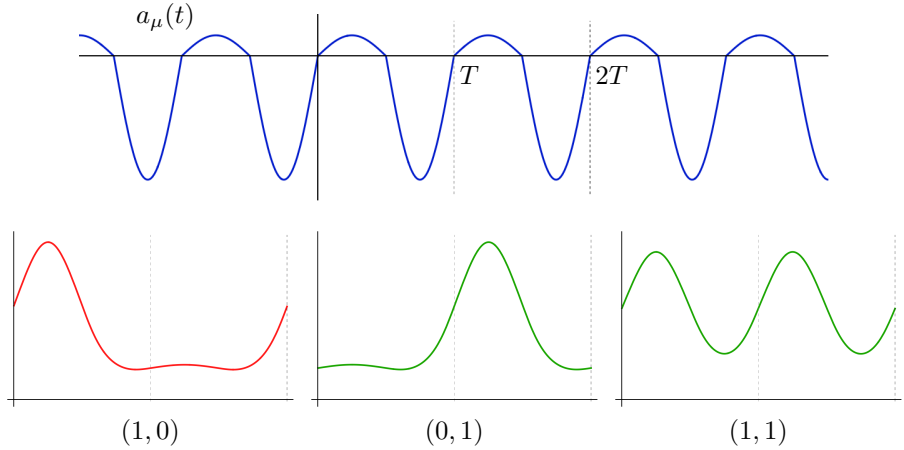

\medskip

As a second step, we prove that  there is at least one subharmonic of order $k$ among the $2^{km}-1$ positive $kT$-periodic solutions of $(\mathscr{E}_{\mu})$ provided by Theorem~\ref{th-sub}.
It is sufficient to consider the string $(1,0,\ldots,0)$ of length $km$ (or equivalently the set of indices $\mathcal{I}:= \{1\}$).
Theorem~\ref{th-sub} ensures the existence of a positive $kT$-periodic solution $u(t)$ to $(\mathscr{E}_{\mu})$ such that $u\in\Lambda^{\{1\}}$.
This implies that there exists $\hat{t}_{1}\in J^{+}_{1,0}=I^{+}_{1}$ such that $r < u(\hat{t}_{1}) < R$ and $0<u(t)<r$ on $J^{+}_{i,\ell}$, for $i=1,\ldots, m$ and $\ell= 0,\ldots, k-1$ with $(i,\ell)\neq(1,0)$. Then $u(\hat{t}_{1})\neq u(t)$ for all $t\in J^{+}_{i,\ell}$, for $i=1,\ldots, m$ and $\ell= 0,\ldots, k-1$ with $(i,\ell)\neq(1,0)$.
Therefore $u(t)$ is not $\ell T$-periodic for every $\ell =1,\ldots,k-1$. We conclude that $u(t)$ is a subharmonic solution of order $k$.
Hence, the proof of Theorem~\ref{th-main2} is concluded.
\qed

\medskip

We complete this section with a discussion on the number of subharmonics. Theorem~\ref{th-main2} ensures the existence of at least a subharmonic solution of order $k$ for every integer $k\geq 0$. Therefore, a natural question that arise is the investigation on the number of subharmonic solutions of order $k$.
Thus now our goal is to select strings of length $km$ which are minimal in some sense (in order to obtain the ``minimality'' of the period of $kT$-periodic solutions) and to count only once the strings corresponding to subharmonics belonging to the same periodicity class.
To this end, we can take advantage of the combinatorial concept of \textit{aperiodic necklaces of length $k$} which are made by arranging $k$ beads whose color is chosen from a list of $n$ colors (see Figure~\ref{fig-02}). Equivalently we can deal also with \textit{$n$-ary Lyndon words of length $k$} (that is strings of $k$ digits of
an alphabet with $n$ symbols which are strictly smaller in the lexicographic ordering than all of its nontrivial rotations).

\definecolor{cGF}{rgb}{1,0.9,0}

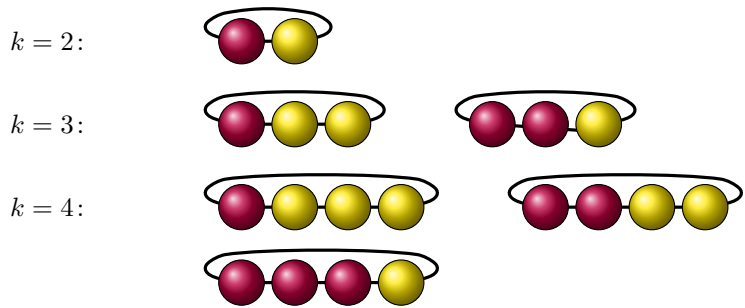
\begin{figure}[h!]
\centering
\begin{tikzpicture}
\draw (-2.5,0) node {$k=2\colon$};
\draw[line width=1.2pt] (0,0) -- (0.7,0);
\draw [line width=1.2pt, rounded corners]
(0.98,0.05) .. controls (1.2,0.1) and (1.2,0.2) .. 
(1.05,0.35)  .. controls (0.5,0.45) and (0.2,0.45) .. 
(-0.35,0.35) .. controls (-0.5,0.2) and (-0.5,0.1) ..
(-0.28,0.05);
\filldraw[ball color=purple] (0,0) circle (0.3);
\filldraw[ball color=cGF] (0.7,0) circle (0.3);
\draw (-2.5,-1.1) node {$k=3\colon$};
\draw[line width=1.2pt] (0,-1.1) -- (1.4,-1.1);
\draw [line width=1.2pt, rounded corners]
(1.68,-1.05) .. controls (1.9,-1) and (1.9,-0.9) .. 
(1.75,-0.75)  .. controls (1.1,-0.65) and (0.3,-0.65) .. 
(-0.35,-0.75) .. controls (-0.5,-0.9) and (-0.5,-1) ..
(-0.28,-1.05);
\filldraw[ball color=purple] (0,-1.1) circle (0.3);
\filldraw[ball color=cGF] (0.7,-1.1) circle (0.3);
\filldraw[ball color=cGF] (1.4,-1.1) circle (0.3);
\draw[line width=1.2pt] (3.3,-1.1) -- (4.7,-1.2);
\draw [line width=1.2pt, rounded corners]
(4.98,-1.05) .. controls (5.2,-1) and (5.2,-0.9) .. 
(5.05,-0.75)  .. controls (4.4,-0.65) and (3.6,-0.65) .. 
(2.95,-0.75) .. controls (2.8,-0.9) and (2.8,-1) ..
(3.02,-1.05);
\filldraw[ball color=purple] (3.3,-1.1) circle (0.3);
\filldraw[ball color=purple] (4,-1.1) circle (0.3);
\filldraw[ball color=cGF] (4.7,-1.1) circle (0.3);
\draw (-2.5,-2.2) node {$k=4\colon$};
\draw[line width=1.2pt] (0,-2.2) -- (2.1,-2.2);
\draw [line width=1.2pt, rounded corners]
(2.38,-2.15) .. controls (2.6,-2.1) and (2.6,-2) .. 
(2.45,-1.85)  .. controls (1.8,-1.75) and (0.3,-1.75) .. 
(-0.35,-1.85) .. controls (-0.5,-2) and (-0.5,-2.1) ..
(-0.28,-2.15);
\filldraw[ball color=purple] (0,-2.2) circle (0.3);
\filldraw[ball color=cGF] (0.7,-2.2) circle (0.3);
\filldraw[ball color=cGF] (1.4,-2.2) circle (0.3);
\filldraw[ball color=cGF] (2.1,-2.2) circle (0.3);
\draw[line width=1.2pt] (4,-2.2) -- (6.1,-2.2);
\draw [line width=1.2pt, rounded corners]
(6.38,-2.15) .. controls (6.6,-2.1) and (6.6,-2) .. 
(6.45,-1.85)  .. controls (5.8,-1.75) and (4.3,-1.75) .. 
(3.65,-1.85) .. controls (3.5,-2) and (3.5,-2.1) ..
(3.72,-2.15);
\filldraw[ball color=purple] (4,-2.2) circle (0.3);
\filldraw[ball color=purple] (4.7,-2.2) circle (0.3);
\filldraw[ball color=cGF] (5.4,-2.2) circle (0.3);
\filldraw[ball color=cGF] (6.1,-2.2) circle (0.3);
\draw[line width=1.2pt] (0,-3.2) -- (2.1,-3.2);
\draw [line width=1.2pt, rounded corners]
(2.38,-3.15) .. controls (2.6,-3.1) and (2.6,-3) .. 
(2.45,-2.85)  .. controls (1.8,-2.75) and (0.3,-2.75) .. 
(-0.35,-2.85) .. controls (-0.5,-3) and (-0.5,-3.1) ..
(-0.28,-3.15);
\filldraw[ball color=purple] (0,-3.2) circle (0.3);
\filldraw[ball color=purple] (0.7,-3.2) circle (0.3);
\filldraw[ball color=purple] (1.4,-3.2) circle (0.3);
\filldraw[ball color=cGF] (2.1,-3.2) circle (0.3);
\end{tikzpicture}
\vspace*{2pt}
\caption{The figure shows the aperiodic necklaces made by arranging $k$ beads whose color is chosen from a list of 2 colors, when $k\in\{2,3,4\}$. The aperiodic necklaces depicted in the figure correspond to the following Lyndon words on the alphabet $\{a,b\}$: $ab$ for $k=2$, $abb$ and $aab$ for $k=3$, $abbb$, $aabb$ and $aaab$ for $k=4$ (which are the minimal elements in the class of equivalence in the lexicographic ordering). For instance, for $k = 4$, note that the string $abab$ is not acceptable as it represents a sequence of period $2$ and the string $bbaa$ is already counted as $aabb$.}
\label{fig-02}
\end{figure}

It is easy to notice that there exists a one-to-one correspondence between the aperiodic necklaces of length $k$ on $n$ colors and the non-null strings of length $k$ on $n$ symbols.
Consequently, in order to obtain an estimate on the number of positive subharmonic solution of order $k$ to equation $(\mathscr{E}_{\mu})$, it is sufficient to compute the number $\mathcal{S}_{n}(k)$ of aperiodic necklaces of length $k$ on $n$ colors when $n=2^{m}$. The number $\mathcal{S}_{n}(k)$ is given by \textit{Witt's formula}:
\begin{equation}\label{Witt}
\mathcal{S}_{n}(k) = \dfrac{1}{k} \sum_{l|k} \mu(l) \, n^{\frac{k}{l}},
\end{equation}
where $\mu(\cdot)$ is the M\"{o}bius function, defined on $\mathbb{N}\setminus\{0\}$ by $\mu(1) = 1$,
$\mu(l) = (-1)^{s}$ if $l$ is the
product of $s$ distinct primes and $\mu(l) = 0$ otherwise.
For instance, the values of $\mathcal{S}_{2}(k)$ (number of binary Lyndon words of length $k$) for $k=2,\ldots,10$ are $1$, $2$, $3$, $6$, $9$, $18$, $30$, $56$, $99$.

\begin{remark}\label{rem-6.3}
Although in this context formula \eqref{Witt} is related to the concepts of aperiodic necklaces and Lyndon words, this formula come out in different areas of mathematics.
Now we provide an overview of the several meanings of \eqref{Witt}.

Still in combinatorics, it is not difficult to see that $\mathcal{S}_{n}(k)$ is also the number of \textit{$n$-ary de Bruijn sequences of order $k$}. A $n$-ary de Bruijn sequence of order $k$ is a circular string of characters
chosen in an alphabet of size $n$,
for which every possible subsequence of length $k$ appears as a substring of consecutive characters exactly once. For example $aaababbb$ and $bbbabaaa$ are the De Bruijn sequence of order $3$ on the size-$2$ alphabet $\{a,b\}$.
For more details about these concepts and other aspects of the formula in the context of combinatorics on words, we refer to \cite{Lo-97, Ma-1891}
and the very interesting historical survey \cite[\S~4]{BePe-07}.

The number $\mathcal{S}_{n}(k)$ has several meanings even outside combinatorics.
For instance, the integer $\mathcal{S}_{2}(k)$ (of binary Lyndon words of length $k$) corresponds to
the number of periodic points with minimal period $k$ in the iteration of the tent map $f(x) := 2 \min\{x,1-x\}$ on the unit interval (cf.~\cite{Du-95},
also for more general formulas) and to the number of distinct cycles of minimal period $k$
in a shift dynamical system associated with a totally disconnected hyperbolic iterated function system (cf.~\cite[Lemma~1, p.~171]{Ba-88}).
Concerning the more general formula for $\mathcal{S}_{n}(k)$, we just mention two other meanings.
The classical \textit{Witt's formula} (proved in 1937), which is still widely studied in algebra,
gives the dimensions of the homogeneous components of degree $k$ of the free
Lie algebra over a finite set with $n$ elements (cf.~\cite[Corollary~5.3.5]{Lo-97}).
Moreover, in Galois theory, $\mathcal{S}_{n}(k)$ is also the number of
monic irreducible polynomials of degree $k$ over the finite field $\mathbb{F}_{n}$,
when $n$ is a prime power (in this context \eqref{Witt} is also known as \textit{Gauss formula};
we refer to \cite[ch.~14, p.~588]{DuFo-04} for a possible proof).

It is not possible to mention here all the other several implications of formula~\eqref{Witt},
for example in symbolic dynamics, algebra, number theory and chaos theory. For this latter topic,
we only recall the recent paper \cite{JoSaYo-10} where such numbers appear in connection with the study of
period-doubling cascades.

Further information and references can be found in \cite{GiRi-61, KoRaWo-14, Sl-oeis}.
$\hfill\lhd$
\end{remark}

As a consequence, the following holds.

\begin{theorem}\label{th-sub2}
Let $a \colon \mathbb{R} \to \mathbb{R}$ be a $T$-periodic locally integrable function of the form \eqref{amu} satisfying $(a_{*})$.
Let $g \in \mathcal{C}^{1}(\mathbb{R}^{+})$ satisfy $(g_{*})$ and $(g_{s})$.
Then there exists $\mu^{*} > 0$ such that, for each $\mu > \mu^{*}$ and each integer $k\geq2$,
equation $(\mathscr{E}_{\mu})$ has at least $\mathcal{S}_{2^{m}}(k)$ positive subharmonic solutions of order $k$.
\end{theorem}

\begin{remark}\label{rem-4.1}
In Remark~\ref{rem-2.1} it is highlighted that Theorem~\ref{th-prel1} and Theorem~\ref{th-prel2} hold also dealing with equations of the form $(\mathscr{E}_{\mu,c})$ which involve a friction term $cu'$.
Since the proof of Theorem~\ref{th-main2}, developed in the present section, is based on these preliminary results and on Theorem~\ref{th-sub}, we stress that Theorem~\ref{th-main2} and Theorem~\ref{th-sub2} are also valid for equation $(\mathscr{E}_{\mu,c})$. We refer to \cite{FeZa-pp2015} for the details.
$\hfill\lhd$
\end{remark}

\begin{remark}\label{rem-4.2}
In this section we have proved how to produce positive subharmonic solutions of arbitrary order $k$, which are ``coded'' by non-null strings of length $k$.
Starting from Theorem~\ref{th-sub}, via a diagonal argument, in \cite{FeZa-pp2015} the authors proved the existence of chaotic dynamics, precisely they show the existence of positive globally defined bounded solutions exhibiting a complex behavior given by a bi-infinite string (similarly as in Theorem~\ref{th-sub}). This is not the aim of the present manuscript and we refer also to \cite{BoZa-12,So-pp2016aims} for some interesting discussion on the subject.
$\hfill\lhd$
\end{remark}

\section{Final comments and open problems}\label{section-5}

We complete the discussion on subharmonic solutions by proposing some open problems that arise from the comparison of the two main results and that will be subjects of future investigations.

\medskip

First of all, we underline that both main theorems provide infinitely many positive subharmonic solutions to equation
\begin{equation*}
u'' + \bigl{(}a^{+}(t) - \mu a^{-}(t)\bigr{)} g(u) = 0.
\leqno{(\mathscr{E}_{\mu})}
\end{equation*}
More precisely, we have the existence of positive subharmonic solution of order $k$ when $\mu>\mu^{\#}$ and $k$ is sufficiently large (by Theorem~\ref{th-main1}) and when $\mu$ is sufficiently large and $k\geq2$ (by Theorem~\ref{th-main2}).

The first natural question is the following.

\begin{quote}
\textit{Open problem 1. } Are the subharmonic solutions obtained in Theorem~\ref{th-main1} and in Theorem~\ref{th-main2} actually distinct?
\end{quote}

It seems difficult to compare solutions obtained with different techniques.
A possible way to handle this problem is to understand the behavior of the solutions, for example in the intervals where $a_{\mu}(t)$ is non-negative. In order to compare solutions, we raised the problem whether it is possible to describe the behavior of the subharmonic solutions obtained in Theorem~\ref{th-main1} in the same manner as we describe the subharmonics in Theorem~\ref{th-main2}.

\medskip

We recall that in Theorem~\ref{th-main1} we obtain subharmonics oscillating around a $T$-periodic solution $u^{*}(t)$ of equation $(\mathscr{E}_{\mu})$ (for $\mu>\mu^{\#}$). 
We also observe that Theorem~\ref{th-prel1} ensures the existence of (at least) $2^{m}-1$ positive $T$-periodic solutions $u^{*}_{i}(t)$ of $(\mathscr{E}_{\mu})$ for $\mu>\mu^{*}$ and any of these solutions is such that $\lambda_{0}\bigl{(}a_{\mu}(t)g'(u^{*}_{i}(t)) \bigr{)}<0$ (by Lemma~\ref{lem-morse}). Therefore, we can apply $2^{m}-1$ times Proposition~\ref{propsub}, with respect to each periodic solution $u^{*}_{i}(t)$ and thus Theorem~\ref{th-main1} is valid with respect to any $u^{*}_{i}(t)$.

\begin{quote}
\textit{Open problem 2. } Are the subharmonic solutions obtained in the manner described above actually distinct?
\end{quote}

\medskip

Finally, as observed in Remark~\ref{rem-4.1}, Theorem~\ref{th-main2} is valid also in the non-variational setting, that is dealing with equation
\begin{equation*}
u'' + cu' + \bigl{(}a^{+}(t) - \mu a^{-}(t)\bigr{)} g(u) = 0,
\leqno{(\mathscr{E}_{\mu,c})}
\end{equation*}
where $c\in\mathbb{R}$ is an arbitrary constant.

\begin{quote}
\textit{Open problem 3. } Can we provide positive subharmonic solutions in the non-variational setting when $\mu>\mu^{\#}$ and $k$ is sufficiently large?
\end{quote}

It is generally expected that the solutions obtained with the symplectic approach described in Section~\ref{section-3} disappear for (even small) perturbations that destroy the Hamiltonian structure. 
Dealing with equation $(\mathscr{E}_{\mu,c})$, we cannot apply the Poincar\'{e}-Birkhoff fixed point theorem as described in Proposition~\ref{propsub}: a more general version of the Poincar\'{e}-Birkhoff theorem is necessary.

\appendix
\section{Mawhin's coincidence degree}\label{appendix-A}

This appendix is devoted to recalling some basic facts about a version
of Mawhin's coincidence degree for open and possibly unbounded sets that is used in the present paper.
For more details about the coincidence degree, proofs and applications, we refer to \cite{GaMa-77,Ma-79,Ma-93} and the references therein.

Let $X$ and $Z$ be real Banach spaces and let
\begin{equation*}
L \colon \text{\rm dom}\,L (\subseteq X) \to Z
\end{equation*}
be a linear Fredholm mapping of index zero, i.e.~$\text{\rm Im}\,L$ is a closed subspace of $Z$ and
$\text{\rm dim}(\ker L) = \text{\rm codim}(\text{\rm Im}\,L)$ are finite.
We denote by $\ker L = L^{-1}(0)$ the kernel of $L$, by $\text{\rm Im}\,L\subseteq Z$ the image of $L$
and by $\text{\rm coker}\,L \cong Z/\text{\rm Im}\,L$ the complementary subspace of $\text{\rm Im}\,L$ in $Z$.
Consider the linear continuous projections
\begin{equation*}
P \colon X \to \ker L, \qquad Q \colon Z \to \text{\rm coker}\,L.
\end{equation*}
so that
\begin{equation*}
X = \ker L \oplus \ker P, \qquad Z = \text{\rm Im}\,L \oplus \text{\rm Im}\,Q.
\end{equation*}
We denote by
\begin{equation*}
K_{P} \colon \text{\rm Im}\,L \to \text{\rm dom}\,L \cap \ker P
\end{equation*}
the right inverse of $L$, i.e.~$L K_{P}(w) = w$ for each $w\in \text{\rm Im}\,L$.
Since $\ker L$ and $\text{\rm coker}\,L$ are finite dimensional vector spaces of the same dimension,
once an orientation on both spaces is fixed, we choose a linear orientation-preserving isomorphism $J \colon \text{\rm coker}\,L \to \ker L$.

Let
\begin{equation*}
N \colon X \to Z
\end{equation*}
be a nonlinear \textit{$L$-completely continuous} operator, namely $N$ and $K_{P}(Id-Q)N$ are continuous, and also $QN(B)$ and $K_{P}(Id-Q)N(B)$ are relatively compact sets, for each bounded set $B\subseteq X$.
For example, $N$ is $L$-completely continuous when $N$ is continuous, maps bounded sets to bounded sets
and $K_{P}$ is a compact linear operator.
Consider the \textit{coincidence equation}
\begin{equation}\label{eq-A.1}
Lu = Nu,\quad u\in \text{\rm dom}\,L.
\end{equation}
One can easily prove that equation \eqref{eq-A.1} is equivalent to the fixed point problem
\begin{equation}\label{eq-A.2}
u = \Phi(u):= Pu + JQNu + K_{P}(Id-Q)Nu, \quad u\in X.
\end{equation}
Notice that, under the above assumptions, $\Phi \colon X \to X$ is a completely continuous operator.
Therefore, applying Leray-Schauder degree theory to the operator equation \eqref{eq-A.2}, it is possible to solve equation \eqref{eq-A.1} when $L$ is not invertible.

\medskip

Let $\mathcal{O}\subseteq X$ be an open and \textit{bounded} set such that
\begin{equation*}
Lu \neq Nu, \quad \forall \, u\in \text{\rm dom}\,L \cap \partial\mathcal{O}.
\end{equation*}
In this case, the \textit{coincidence degree of $L$ and $N$ in $\mathcal{O}$} is defined as
\begin{equation*}
D_{L}(L-N,\mathcal{O}):= \text{\rm deg}_{LS}(Id-\Phi,\mathcal{O},0),
\end{equation*}
where ``$\text{\rm deg}_{LS}$'' denotes the Leray-Schauder degree. We underline that $D_{L}$ is independent
on the choice of the linear orientation-preserving isomorphism $J$ and of the projectors $P$ and $Q$.

\medskip

Now we present an extension of the coincidence degree to open (possibly unbounded) sets, following the
standard approach used in the theory of fixed point index
to define the Leray-Schauder degree for locally compact maps on arbitrary open sets (cf.~\cite{Nu-85,Nu-93}).
Extensions of coincidence degree to the case of general open sets have been already considered in
previous articles (see for instance \cite{CaHeMaZa-94, MaReZa-00, Mo-96}).

Consider an open set $\Omega\subseteq X$ and suppose that the solution set
\begin{equation*}
\text{\rm Fix}\,(\Phi,\Omega):= \bigl{\{}u\in {\Omega} \colon u =
\Phi u\bigr{\}} = \bigl{\{}u\in {\Omega}\cap \text{\rm dom}\,L \colon Lu = N u\bigr{\}}
\end{equation*}
is compact. According to the extension of Leray-Schauder degree, we can define
\begin{equation*}
\text{\rm deg}_{LS}(Id - \Phi,\Omega,0):=\text{\rm deg}_{LS}(Id - \Phi,\mathcal{V},0),
\end{equation*}
where $\mathcal{V}$ is an open and bounded set with $\text{\rm Fix}\,(\Phi,\Omega) \subseteq \mathcal{V} \subseteq \overline{\mathcal{V}} \subseteq \Omega$.
The definition is independent of the choice of $\mathcal{V}$.
In this case the \textit{coincidence degree} \textit{of $L$ and $N$ in $\Omega$} is defined as
\begin{equation*}
D_{L}(L-N,\Omega):= D_{L}(L-N,{\mathcal{V}}) =\text{\rm deg}_{LS}(Id - \Phi,\mathcal{V},0),
\end{equation*}
with $\mathcal{V}$ as above.
Using the excision property of the Leray-Schauder degree, it is easy to check that
if $\Omega$ is an open and bounded set satisfying $Lu \neq Nu$, for all $u\in \partial\Omega\cap \text{\rm dom}\,L$,
this definition is the usual definition of coincidence degree described above.

\medskip

The main properties of the coincidence degree are the following.
\begin{itemize}
\item \textit{Additivity. }
Let $\Omega_{1}$, $\Omega_{2}$ be open and disjoint subsets of $\Omega$ such that $\text{\rm Fix}\,(\Phi,\Omega)\subseteq \Omega_{1}\cup\Omega_{2}$.
Then
\begin{equation*}
D_{L}(L-N,\Omega) = D_{L}(L-N,\Omega_{1})+ D_{L}(L-N,\Omega_{2}).
\end{equation*}
\item \textit{Excision. }
Let $\Omega_{0}$ be an open subset of $\Omega$ such that $\text{\rm Fix}\,(\Phi,\Omega)\subseteq \Omega_{0}$.
Then
\begin{equation*}
D_{L}(L-N,\Omega)=D_{L}(L-N,\Omega_{0}).
\end{equation*}
\item \textit{Existence theorem. }
If $D_{L}(L-N,\Omega)\neq0$, then $\text{\rm Fix}\,(\Phi,\Omega)\neq\emptyset$,
hence there exists $u\in {\Omega}\cap \text{\rm dom}\,L$ such that $Lu = Nu$.
\item \textit{Homotopic invariance. }
Let $H\colon\mathopen{[}0,1\mathclose{]}\times \Omega \to X$, $H_{\vartheta}(u) := H(\vartheta,u)$, be a continuous homotopy such that
\begin{equation*}
\mathcal{S}:=\bigcup_{\vartheta\in\mathopen{[}0,1\mathclose{]}} \bigl{\{}u\in \Omega\cap \text{\rm dom}\,L \colon Lu=H_{\vartheta}u\bigr{\}}
\end{equation*}
is a compact set and there exists an open neighborhood $\mathcal{W}$ of $\mathcal{S}$ such that $\overline{\mathcal{W}}\subseteq \Omega$ and
$(K_{P}(Id-Q)H)|_{\mathopen{[}0,1\mathclose{]}\times\overline{\mathcal{W}}}$ is a compact map.
Then the map $\vartheta\mapsto D_{L}(L-H_{\vartheta},\Omega)$ is constant on $\mathopen{[}0,1\mathclose{]}$.
\end{itemize}

The following two results are of crucial importance in the computation of the coincidence degree in open and bounded sets.
The proofs are omitted.
Lemma~\ref{lemma_Mawhin} is taken from \cite[Theorem~IV.1]{GaMa-77} and \cite[Theorem~2.4]{Ma-93}, while
Lemma~\ref{lem-abs-deg0} is a classical result (see \cite{Nu-73}) adapted to the present setting
as in \cite[Lemma~2.2, Lemma~2.3]{FeZa-15ade}. By ``$\text{\rm deg}_{B}$'' we denote the Brouwer degree.

\begin{lemma}\label{lemma_Mawhin}
Let $L$ and $N$ be as above and let $\Omega\subseteq X$ be an open and bounded set.
Suppose that
\begin{equation*}
Lu \neq \vartheta Nu, \quad \forall \, u\in \text{\rm dom}\,L \cap \partial\Omega, \; \forall \, \vartheta\in\mathopen{]}0,1\mathclose{]},
\end{equation*}
and
\begin{equation*}
QN(u)\neq0, \quad \forall\, u\in \partial\Omega \cap \ker L.
\end{equation*}
Then
\begin{equation*}
D_{L}(L-N,\Omega) = \text{\rm deg}_{B}(-JQN|_{\ker L},\Omega \cap \ker L,0).
\end{equation*}
\end{lemma}

\begin{lemma}\label{lem-abs-deg0}
Let $L$ and $N$ be as above and let $\Omega\subseteq X$ be an open and bounded set.
Suppose that there exist a vector $v\neq0$ and a constant $\alpha_{0} > 0$ such that
\begin{equation*}
Lu \neq Nu + \alpha v,
\quad \forall \, u\in \text{\rm dom}\,L \cap \partial\Omega, \; \forall\, \alpha\in \mathopen{[}0,\alpha_{0}\mathclose{]},
\end{equation*}
and
\begin{equation*}
Lu \neq Nu + \alpha_{0} v, \quad \forall \, u\in \text{\rm dom}\,L \cap \Omega.
\end{equation*}
Then
\begin{equation*}
D_{L}(L-N,\Omega) = 0.
\end{equation*}
\end{lemma}

Finally we state and prove a key lemma for the computation of the degree in open and unbounded sets.
This result is a more general version of Lemma~\ref{lem-abs-deg0}.

\begin{lemma}\label{lem-deg0-deFigueiredo}
Let $L$ and $N$ be as above and let $\Omega\subseteq X$ be an open set.
Suppose that there exist a vector $v\neq0$ and a constant $\alpha_{0} > 0$ such that
\begin{itemize}
 \item [$(a)$] $Lu \neq Nu + \alpha v$, for all $u\in \text{\rm dom}\,L \cap \partial\Omega$ and for all $\alpha\geq0$;
 \item [$(b)$] for all $\beta\geq0$ there exists $R_{\beta}>0$ such that if
 there exist $u\in \overline{\Omega}\cap \text{\rm dom}\,L$ and $\alpha\in\mathopen{[}0,\beta\mathclose{]}$ with
             $Lu = Nu + \alpha v$, then $\|u\|_{X}\leq R_{\beta}$;
 \item [$(c)$] there exists $\alpha_{0}>0$ such that $Lu \neq Nu + \alpha v$, for all $u\in \text{\rm dom}\,L \cap \Omega$ and $\alpha\geq\alpha_{0}$.
\end{itemize}
Then
\begin{equation*}
D_{L}(L-N,\Omega) = 0.
\end{equation*}
\end{lemma}

\begin{proof}
For $\alpha\geq 0$, let us consider the set
\begin{equation*}
\mathcal{R}_{\alpha}
:=\bigl{\{} u\in \overline{\Omega}\cap \text{\rm dom}\,L \colon Lu = Nu + \alpha v\bigr{\}}
= \bigl{\{} u\in \overline{\Omega} \colon u = \Phi u + \alpha v^{*}\bigr{\}},
\end{equation*}
where $v^{*} := JQv + K_{P}(Id-Q)v$.
Without loss of generality, we assume that $R_{\alpha'}<R_{\alpha''}$ for $\alpha'<\alpha''$.
By conditions $(a)$, for all $\alpha\geq0$, the solution set $\mathcal{R}_{\alpha}$ is disjoint from $\partial\Omega$.
Moreover, by conditions $(b)$ and $(c)$, $\mathcal{R}_{\alpha}$ is contained in $\Omega\cap B(0,R_{\alpha_{0}+1})$.
So $\mathcal{R}_{\alpha}$ is bounded, and hence compact.
In this manner we have proved that the coincidence degree $D_{L}(L-N-\alpha v,\Omega)$ is well defined for any $\alpha\geq 0$.

Now, condition $(c)$, together with the property of existence of solutions when the degree $D_{L}$ is nonzero,
implies that there exists $\alpha_{0}\geq 0$ such that
\begin{equation*}
D_{L}(L-N- \alpha_{0} v,\Omega)=0.
\end{equation*}
On the other hand, from condition $(b)$ applied to $\beta = \alpha_{0}$,
repeating the same argument as above, we find that the set
\begin{equation*}
\mathcal{S}
:= \bigcup_{\alpha\in \mathopen{[}0,\alpha_{0}\mathclose{]}} \mathcal{R}_{\alpha} \,
= \bigcup_{\alpha\in \mathopen{[}0,\alpha_{0}\mathclose{]}} \bigl{\{}u\in \overline{\Omega}\cap \text{\rm dom}\,L \colon Lu = N u + \alpha v\bigr{\}}
\end{equation*}
is a compact subset of $\Omega$. Hence, by the homotopic invariance of the coincidence degree, we have that
\begin{equation*}
D_{L}(L-N,\Omega) = D_{L}(L-N - \alpha_{0} v,\Omega) = 0.
\end{equation*}
This concludes the proof.
\end{proof}

\section*{Acknowledgments}

This work benefited from several fruitful discussions with Alberto Boscaggin and Fabio Zanolin, and from helpful suggestions of Elisa Sovrano.
The author was also grateful to Pasquale Candito, Giuseppina D'Agu\`{i}, Roberto Livrea, Salvatore A. Marano, Sunra Mosconi, for the invitation and the perfect hospitality both in the special session SS92
of the AIMS Conference in Orlando (July 1--5, 2016) and in the workshop in Catania (October 28--29, 2016). Indeed this survey paper is based on a talk given by the author in Catania.

The author wrote the present manuscript in his first month at University of Mons and he is grateful to the Department of Mathematics for the pleasant welcome, in particular to Christophe Troestler, and to Aline Goulard, Marie-Aurore Mainil, Monia Mestiri.

\bibliographystyle{elsart-num-sort}
\bibliography{Feltrin_biblio}

\bigskip
\begin{flushleft}

{\small{\it Preprint}}

{\small{\it January 2017}}

\end{flushleft}

\end{document}